\documentclass[10pt]{amsart}

\usepackage{graphicx}
\usepackage{mathtools}

\allowdisplaybreaks
\usepackage{enumerate,amssymb} 
\usepackage{mathrsfs} 
\usepackage{tikz} 
\usepackage{verbatim}

\newtheorem{theorem}{Theorem}[section]
\newtheorem{lemma}[theorem]{Lemma}

\theoremstyle{definition}

\newtheorem{proposition}[theorem]{Proposition}
\newtheorem{corollary}[theorem]{Corollary}

\theoremstyle{remark}
\newtheorem{remark}[theorem]{Remark}

\numberwithin{equation}{section}

\usepackage{todonotes}

\DeclareMathOperator{\dv}{div}
\DeclareMathOperator{\grad}{grad}
\DeclareMathOperator{\Ric}{Ric}
\DeclareMathOperator{\inj}{Inj}

\DeclareMathOperator{\dist}{dist}
\DeclareMathOperator{\dvol}{\,dvol}
\DeclareMathOperator{\dS}{\,dS}

\begin{document}

\title{Eigenvalues on Spherically Symmetric Manifolds}

\author[S.~M.~Berge]{Stine Marie Berge}
\address{Institute for Analysis, Leibniz University Hannover, Welfengarten 1 30167 Hannover, Germany}
\email{stine.berge@math.uni-hannover.de}



\date{}

\dedicatory{}

\commby{}
\begin{abstract}
In this article we will explore Dirichlet Laplace eigenvalues on balls on spherically symmetric manifolds. We will compare any Dirichlet Laplace eigenvalue with the corresponding Dirichlet Laplace eigenvalue on balls in Euclidean space with the same radius. As a special case we will show that the Dirichlet Laplace eigenvalues on balls with small radius on the sphere are smaller than the corresponding eigenvalues on the Euclidean ball with the same radius. While the opposite is true for the Dirichlet Laplace eigenvalues of hyperbolic spaces.
\end{abstract}

\maketitle


\section{Introduction} 
It is well know that the Dirichlet Laplace eigenvalues of the ball $B_{r_0}(0)\subset \mathbb{R}^n$ with radius $r_0$ are on the form 
\[\left(\frac{j^l_{m+n/2-1}}{r_0}\right)^2,\]
where $j^l_{m+n/2-1}$ is the $l$'th zero of the Bessel function $J_{m+n/2-1}$.
In \cite{Ba90, Ba91} Baginski compared all the Dirichlet Laplace eigenvalues on the spherical cap $B_{r_0}(p)\subset \mathbb{S}^{2}$ to the Dirichlet Laplace eigenvalues of the ball in the plane with the same radius. In particular, the author showed that the first eigenvalue $\lambda_1(r_0)$ on $B_{r_0}(p)\subset \mathbb{S}^2$ satisfies
\begin{equation}\label{eq:first_eigenvalue}
\left(\frac{j_0^1}{r_0}\right)^2-\frac{1}{4}\left(1+\frac{1}{\sin^2(r_0)}-\frac{1}{r_0^2}\right)\le \lambda_1(r_0)\le \left(\frac{j_0^1}{r_0}\right)^2-\frac{1}{3}. 
\end{equation}
One of the tools used in the proof are comparison theorems for Sturm-Liouville equations on the radial part of the eigenfunction.

More recently, the first eigenvalue of a ball on a spherically symmetric, (also called rotationally symmetric) manifold was compared to the first eigenvalue of a ball in Euclidean space in \cite[Lem.~3.1]{BF17}. When applied to hyperbolic space and the sphere the result gives a generalization of \eqref{eq:first_eigenvalue} for the first eigenvalue found in \cite[Thm.~3.3]{BF17}. As a general consequence of this result, Borisov and Freitas showed that the first Dirichlet Laplace eigenvalue on balls on hyperbolic space are larger than the first eigenvalue on Euclidean space. Prior, a lower bound in the case of hyperbolic space for the first Dirichlet Laplace eigenvalue on balls has been shown by Artamoshin, see \cite{Ar16}. In the same article, the author also computed all the eigenvalues corresponding to radial eigenfunctions in $3$-dimensional hyperbolic space.

In this article, we will present two different proofs of similar inequalities to \eqref{eq:first_eigenvalue} for all eigenvalues for both the spheres and hyperbolic spaces. Both the proofs will heavily dependent on the existence of certain continuous families of eigenfunctions and corresponding eigenvalues. The construction of these continuous families will for constant curvature spaces be outlined in Section  \ref{sec:model_space_eigenvalues}, and generalized to spherically symmetric manifolds in Section \ref{sec:Spherical Symmetric Manifolds}. This parametrized family of Dirichlet eigenvalues on the geodesic ball $B_r(p)$ with radius $r$ centered at $p$ will be denoted by $\lambda_{m,l}(r)$ and satisfies 
\[\lim_{r\to 0}r^2\lambda_{m,l}(r)=(j_{m+n/2-1}^l)^2.\]
We will show that these families of eigenvalues satisfies a lower and upper bound similar to \eqref{eq:first_eigenvalue}.
As a corollary, we will show that for every eigenvalue there exists a continuous family of Dirichlet Laplace eigenvalues $\lambda_{m,l}(r)$  on hyperbolic space with curvature $K$ on the geodesic ball $B_r(p)$, such that \[\lim_{r\to \infty}\lambda_{m,l}(r)=-\left(\frac{n-1}{2}\right)^2K.\] For the first eigenvalue corresponding to the case $m=0$ and $l=1$ this was previously shown by Randol, see \cite[p.~46]{Ch84}.

The outline of the paper is as follows: In Section \ref{sec:model_space_eigenvalues} we will prove a version of \eqref{eq:first_eigenvalue} for all eigenvalues using Sturm-Liouville theory as outlined in \cite{Ba90,Ba91}. We will go through the definition and some theory about spherical symmetric manifolds in Section \ref{sec:spherically_symmetric}. Of special importance for this paper, is representing eigenfunctions of balls on spherically symmetric manifolds using spherical harmonics. The goal of Section \ref{sec:eigenvalues_balls} will be to express the eigenvalues as zeros of radial functions satisfying a Sturm-Liouville equation. 

In the last section, we will write the an family of eigenvalue as an integral of the corresponding eigenfunction. To be more precise, let $u^t$ be the normalized continuous family of Dirichlet Laplace eigenfunction satisfying $\Delta u^t=\lambda_{m,l}(t)u^t$ on $B_t(p)$ in a spherically symmetric manifold of dimension $n$. Using this notation we the eigenvalue $\lambda_{m,l}(r_1)$ can be written as the integral
   \begin{equation}\label{eq:radial eigenvalue eq}
   \lambda_{m,l}(r_1)=\left(\frac{j^l_{m+n/2-1}}{r_1}\right)^2+\frac{1}{2r_1^2}\int_{0}^{r_1}t\int_{B_{t}(p)}(u^t)^{2}(x)g(r(x))\dvol \mathrm{dt},
   \end{equation}
where $g$ is a radial function only depending on the geometry. For the eigenvalue $\lambda_{0,1}(r)$ this result can be found in \cite[Lem.~3.1]{BF17}. As an application of formula \eqref{eq:radial eigenvalue eq} we give another proof of the theorem presented in Section \ref{sec:model_space_eigenvalues} for eigenvalues on spheres and hyperbolic space. 

\section{Dirichlet Laplace Eigenvalues on Model Spaces}\label{sec:model_space_eigenvalues}
In this section we will work on the model space $(M_K, \mathbf{g}_K)$ with constant sectional curvature $K$ and dimension $n$. Denote by $\sin_K(r)$ the function 
\[\sin_K(r)=
\begin{cases}
   \sin(r\sqrt{K})/\sqrt{K} & \text{for } K>0\\
   r & \text{for } K=0\\
   \sinh(r\sqrt{-K})/\sqrt{-K} & \text{for } K<0,
\end{cases}
   \]
   and $\cos_K(r)=\sin_K(r)'$.
   For a fixed point $p\in M_K$ denote by $r_p(x)=\dist(p,x)$ the radial distance function and let $\mathbb{S}^{n-1}_1\subset \mathbb{R}^{n}$ denote the $(n-1)$-sphere of radius $1$. To represent points in $M_K$ we will use geodesic spherical coordinates $(r,\theta)$ where $\theta\in \mathbb{S}_1^{n-1}$. Geodesic coordinates are defined by using the exponential map, see \cite[Sec.~III.1.]{Ch06}. Define $B_{r_0}(p)=r_p^{-1}([0,r_0))$ and $S_{r_0}^{n-1}(p)=\partial B_{r_0}(p)$.

   In Section \ref{sec:spherically_symmetric} we will show that the solutions to the Dirichlet Laplace eigenvalue problem 
\begin{equation}\label{eq:Dirichlet_modelspaces}
\begin{cases}
   \Delta u^{r_0}+\lambda(r_0) u^{r_0}=0&\text{in }B_{r_0}(p)\text{ for $p\in M_K$,}\\
   u^{r_0}=0&\text{at }S_{r_0}^{n-1}(p)
\end{cases}
\end{equation}
that are given on the form $u^{r_0}(r,\theta)=R_m^{r_0}(r)\Theta_m(\theta)$ form a basis for $L^2(B_{r_0}(p))$. Using separation of variables we can show that $\Theta_m$ is a spherical harmonic function solving \[\Delta_{\mathbb{S}_1^{n-1}}\Theta_m+m(m+n-2)\Theta_m=0,\]
where $m\in \mathbb{N}$. The radial part $R_m^{r_0}$ on the other hand solves the equation 
   \begin{multline}\label{eq:radial_spherical}
      R_m^{r_0}(r)''+ (n-1)\frac{\cos_K(r) }{\sin_K(r)}R_m^{r_0}(r)'
       - \frac{m(m+n-2)}{\sin_K^2(r)}R_m^{r_0}(r)=-\lambda(r_0) R_m^{r_0}(r)
   \end{multline}
   with the condition $R_m^{r_0}(r_0)=0$ and $R_m^{r_0}(0)$ is bounded. Notice that \eqref{eq:radial_spherical} makes $R_m^{r_0}$ to an eigenfunction of the second order differential operator      
   \[ \frac{d^2}{dr^2}+ (n-1)\frac{\cos_K(r) }{\sin_K(r)}\frac{d}{dr} - \frac{m(m+n-2)}{\sin_K^2(r)}.\]
   Denote by $R_{m,l}^{r_0}$ the eigenfunction to \eqref{eq:radial_spherical} corresponding to the $l$'th eigenvalue when ordered by size. The corresponding eigenvalue will be denoted by $\lambda_{m,l}(r_0)$.
   It is known, see e.g.\ \cite[p.~318]{Ch84}, that 
   \[\lim_{r_0\to 0}r_0^2\lambda_{m,l}(r_0)=(j^l_{m+n/2-1})^2,\]
   where $j^l_{m+n/2-1}$ is the $l$'th positive zero of the Bessel function $J_{m+n/2-1}$.

We are now ready to state the main theorem:
\begin{theorem}\label{thm:radial:model:space}
   Let $\lambda_{m,l}(r_0)$ be defined as above. Furthermore, when $K>0$ we will assume that $r_0<\pi/\sqrt{K}$. In the case that either $m>0$ or $n>2$ we have that 
   \begin{align*}
      &\left(\frac{j^l_{m+n/2-1}}{r_0}\right)^2+\frac{2m^2+2m(n-2)-n(n-1)}{6}K\le \lambda_{m,l}(r_0)\\
      &\qquad\qquad\le \left(\frac{j^l_{m+n/2-1}}{r_0}\right)^2-\left(\frac{n-1}{2}\right)^2K+ \left(\left(\frac{n-2}{2}+m\right)^2-\frac{1}{4}\right)\left(\frac{1}{\sin_K^2(r_0)}-\frac{1}{r_0^2}\right).
   \end{align*}
   When $m=0$ and $n=2$ we get 
   \[\left(\frac{j^l_{0}}{r_0}\right)^2-\frac{1}{4}\left(K+\frac{1}{\sin_K^2(r_0)}-\frac{1}{r_0^2}\right)\le \lambda_{0,l}(r_0)\le \left(\frac{j^l_{0}}{r_0}\right)^2-\frac{1}{3}K.\]
\end{theorem}

\begin{remark}\hfill
   \begin{itemize}
      \item In the case of the sphere $\mathbb{S}^2$ this inequality is known from \cite{Ba90, Ba91}. Additionally, for the first eigenvalue the theorem is known from \cite[Thm.~3.3]{BF17}.
      \item When $K<0$ we have that \[  \frac{1}{\sin_K^2(r)}-\frac{1}{r^2}\le 0 \quad \text{ and } \quad \lim_{r\to \infty}\frac{1}{\sin_K^2(r)}-\frac{1}{r^2}=0.\]
      When $K > 0$ we have that \[\lim_{r\to \frac{\pi}{\sqrt{K}}}\frac{1}{\sin_K^2(r)}-\frac{1}{r^2}=\infty.\]
      For both positive and negative curvature 
      \[\lim_{r\to 0} \frac{1}{\sin_K^2(r)}-\frac{1}{r^2} = \frac{K}{3}.\]
      \item For $n=3$ and $m=0$ Theorem \ref{thm:radial:model:space} simplifies to  
   \begin{equation*}
   \lambda_{0,l}(r_0)= \left(\frac{j^l_{1/2}}{r_0}\right)^2-K.
   \end{equation*}
   For hyperbolic space this equality was shown in \cite{Ar16}.
   \end{itemize}
\end{remark}

To compare the eigenfunctions of \eqref{eq:radial_spherical} to the Bessel equation we are going to use the Sturm-Picone comparison theorem.
\begin{theorem}[Sturm-Picone Comparison Theorem, {\cite[Theorem B]{Hi05}}]\label{D:T:sturm-comparison}
	Let $y_1$ and $y_2$ be non-zero solutions to 
		\begin{align*}
			(p_1(x)y_1'(x))'+q_1(x)y_1(x)&=0, \\
			(p_2(x)y_2'(x))'+q_2(x)y_2(x)&=0,
		\end{align*}
		on the interval $[a,b]$.
		Assume that $0<p_2\le p_1$ and $q_1\le q_2$, and let $z_1$ and $z_2$ be two consecutive zeros of $y_1$. Then either $y_2$ has a zero in the interval $(z_1,z_2)$, or $y_1=y_2$.
\end{theorem}

\begin{proof}[Proof Thm.~\ref{thm:radial:model:space}]
In this section we will assume that for $K>0$ we have $r_0 <\frac{\pi}{2\sqrt{|K|}}$. For the case when $r_0 \ge\frac{\pi}{2\sqrt{|K|}}$ see the proof in Section \ref{ex:model_eigenvalues}.

To get \eqref{eq:radial_spherical} on Sturm-Liouville form we will define
\[v_K(r)\coloneqq \sin_{K}^{\frac{n-1}{2}}(r)R_{m,l}^{r_0}(r),\]
and let $\tilde{m}\coloneqq\frac{n-2}{2}+m$.
Then we have that 
\begin{equation}\label{eq:curvature_comp}
    0=v_K''(r)+ \left(\left(\frac{n-1}{2}\right)^2K+\lambda_{m,l}(r_0)+\frac{1-4\tilde{m}^2}{4\sin_K^2(r)}\right)v_K(r).
\end{equation}
Let $J_{\tilde{m}}$ be a Bessel function of order $\tilde{m}$'th, and define $u_{C}(r)\coloneqq\sqrt{r}J_{\tilde{m}}(Cr)$ to be the solution to 
\begin{equation}\label{eq:thm_bessel}
0=u_{C}''(r)+\left(C^2+\frac{1-4\tilde{m}^2}{4r^2}\right)u_{C}(r).
\end{equation}
Our goal is to apply Sturm-Picone comparison theorem (Thm.~\ref{D:T:sturm-comparison}) to \eqref{eq:curvature_comp} and \eqref{eq:thm_bessel}.
We will use Sturm-Picone comparison twice with the $C$ in \eqref{eq:thm_bessel} being two different constants which we will denote by $C_1$ and $C_2$.

Notice that $\frac{1}{\sin_{K}^2(r)}-\frac{1}{r^2}$ is increasing with the lower limit 
\[\lim_{r\to 0^+}\frac{1}{\sin_{K}^2(r)}-\frac{1}{r^2}=\frac{K}{3}.\]
 Denote by 
\[
a_{\tilde{m}}\coloneqq \frac{4\tilde{m}^2-1}{4}\left(\frac{1}{\sin_{K}^2(r_0)}-\frac{1}{r_0^2}\right),\qquad b_{\tilde{m}}\coloneqq \frac{K(4\tilde{m}^2-1)}{12}.\]
For the next part we will need a lower bound for the eigenvalue.
We will show that for $K>0$ and $m>0$  we have that
\begin{equation}\label{eq:eigenvalue}
\lambda_{m,l}(r_0)>\frac{m(m+n-2)}{\sin_K^2(r_0)}.
\end{equation}
When $m>0$ we have that the first zero $r_0$ of $R_{m,l}^{r_0}$ is occurs after the first extremal point $r_{\mathrm{max}}<r_0$. Without loss of generality, we can assume that the first extremal point is a maximum. At the point $r_{\mathrm{max}}$ we have that ${R_{m,l}^{r_0}}''(r_{\mathrm{max}})<0$ and ${R_{m,l}^{r_0}}'(r_{\mathrm{max}})=0$, which by using \eqref{eq:radial_spherical} implies
   \begin{multline*}
      (\lambda_{m,l}(r_0) \sin_K^2(r_0) - m(m+n-2))R_{m,l}^{r_0}(r)\\>(\lambda_{m,l}(r_0) \sin_K^2(r) - m(m+n-2))R_{m,l}^{r_0}(r)>0,
   \end{multline*}
   hence \eqref{eq:eigenvalue} follows.
For the negatively curved spaces we will use that 
\[-\frac{(n-1)^2}{4}K< \lambda_{m,l}(r_0)\]
which can be found e.g.~\cite[p.~46]{Ch84}.

When either $n>2$ or $m>0$ we have that $1-4\tilde{m}^2\le 0$, and we will set 
\[C_1=\sqrt{\lambda_{m,l}(r_0)+\left( \frac{n-1}{2} \right)^2K-a_{\tilde{m}}}\]
and 
\[C_2=\sqrt{\lambda_{m,l}(r_0)+\left( \frac{n-1}{2} \right)^2K-b_{\tilde{m}}}.\]
In this case we have that 
\[C_1^2 +\frac{1-4\tilde{m}^2}{4r^2}\le\left(\frac{n-1}{2}\right)^2K+\lambda_{m,l}(r_0)+\frac{1-4\tilde{m}^2}{4\sin_K^2(r)}\le C_2^2 +\frac{1-4\tilde{m}^2}{4r^2}.\]
Using Theorem \ref{D:T:sturm-comparison} for solutions to the equations \eqref{eq:curvature_comp} and \eqref{eq:thm_bessel} leads to the estimate of $r_0$ by 
\[\frac{j^l_{\tilde{m}}}{C_2}\le r_0\le \frac{j^l_{\tilde{m}}}{C_1}.\]
Solving for $\lambda_{m,l}(r_0)$ we get 
\begin{align*}
&\left(\frac{j^l_{\tilde{m}}}{r_0}\right)^2+\frac{2m^2+2m(n-2)-n(n-1)}{6}K\le \lambda_{m,l}(r_0)\\
&\qquad\qquad\le \left(\frac{j^l_{\tilde{m}}}{r_0}\right)^2-\left(\frac{n-1}{2}\right)^2K+ \frac{4\tilde{m}^2-1}{4}\left(\frac{1}{\sin_K^2(r_0)}-\frac{1}{r_0^2}\right).
\end{align*}

When $m=0$ and $n=2$ we have that 
\[C_2^2 +\frac{1-4\tilde{m}^2}{4r^2}\le\left(\frac{n-1}{2}\right)^2K+\lambda_{0,l}(r_0)+\frac{1-4\tilde{m}^2}{4\sin_K^2(r)}\le C_1^2 +\frac{1-4\tilde{m}^2}{4r^2}.\]
Hence we obtain 
\[\left(\frac{j^l_{0}}{r_0}\right)^2-\frac{1}{4}\left(K+\frac{1}{\sin_K^2(r_0)}-\frac{1}{r_0^2}\right)\le \lambda_{0,l}(r_0)\le \left(\frac{j^l_{0}}{r_0}\right)^2-\frac{1}{3}K.\qedhere\]
\end{proof}

A corollary of Theorem \ref{thm:radial:model:space} is as follows:
\begin{corollary}\label{cor:D:2}
Assume that $K<0$. Then in the notation of Theorem \ref{thm:radial:model:space} we have that
\begin{equation}\label{eq:limit}
\lim_{r_0\to \infty}\lambda_{m,l}(r_0)= -\left(\frac{n-1}{2}\right)^2K.
\end{equation}
\end{corollary}
\begin{proof}
   By \cite[p.~46]{Ch84} we know that the smallest eigenvalue $\lambda_{0,1}(r_0)$ satisfies
   \[\lambda_{0,1}(r_0)\ge -\left(\frac{n-1}{2}\right)^2K.\]
   Using this lower bound together with taking the limit as $r_0$ goes towards infinity of the upper bound in Theorem \ref{thm:radial:model:space} gives the result for the case when either $m>0$ or $n>2$. Hence we are only left with the case when $n=2$ and $m=0$. Let us assume that $n=2$ and consider the operators 
   \[(L_0\phi)(r)=\frac{(\sin_K(r)\phi'(r))'}{\sin_K(r)}\quad\text{and}\quad (L_1\phi)(r)=\frac{(\sin_K(r)\phi'(r))'}{\sin_K(r)}-\frac{\phi(r)}{\sin_K^2(r)}.\]
   Then both $L_0$ and $L_1$ are negative and symmetric operators on the $L^2$-space 
      \[X_{r_0}= \left\{\phi \in H^2_{\mathrm{loc}}(0,r_0):\int_{0}^{r_0}\sin_K(r)\phi(r)^2\,\mathrm{dr}<\infty,\, \phi(r_0)=\phi(0)\phi'(0)=0\right\},\]
      with the norm \[\langle\phi, \xi\rangle_{\sin_K} = \int_{0}^{r_0}\phi(r)\overline{\xi(r)}\sin_K(r)\,\mathrm{dr}.\]
      Then for $\phi\in X_{r_0}$ one have the inequality 
      \[\langle L_0\phi, \phi\rangle_{\sin_K} \le\langle L_1\phi, \phi\rangle_{\sin_K}.\]
      This implies that the eigenvalues satisfy 
      \[\lambda_{0,l}(r_0)\le \lambda_{1,l}(r_0).\]
      Taking the limit as $r_0\to \infty$ gives 
      \[\lim_{r_0\to \infty}\lambda_{0,l}(r_0)\le\lim_{r_0\to \infty}\lambda_{1,l}(r_0)=-\left(\frac{n-1}{2}\right)^2K,\]
      proving the claim.
\end{proof}
\begin{remark}\hfill
\begin{enumerate}
    \item  For the smallest Dirichlet eigenvalue on hyperbolic space, i.e.\ $m=0$ and $l=1$, Corollary \ref{cor:D:2} is well known, see e.g.\ \cite[p.~46]{Ch84}. 
  \item There exist similar inequalities for eigenvalues for more general manifolds. One example of similar inequalities is given in e.g.\ \cite[Cor.~2.3]{Ch75}. As noted in \cite[Re.~1.3]{KLP21}, given a complete simply connected Riemannian manifold $(M,\mathbf{g})$ with Ricci curvature satisfying $\Ric\ge (n-1)(-\kappa)$ for $\kappa>0$ we have that 
  \[\lambda_k(B_r(x))\le \frac{(n-1)^2}{4}\kappa+\frac{c_nk^2}{r^2}.\]
  This gives the inequality 
  \[\lim_{r_0\to \infty}\lambda_k(r_0)\le \frac{(n-1)^2}{4},\]
  for $\mathbb{H}^n$ with constant curvature $-1$.
\end{enumerate}
\end{remark}
\section{Eigenvalues on Spherically Symmetric Manifolds}\label{sec:spherically_symmetric}
Let $(M,\mathbf{g})$ be a Riemannian manifold of dimension $n$. Fix a point $p\in M$ and use the notations $r_p(x)=\dist(x,p)$ and $S_{r_0}^{n-1}=r_p^{-1}(r_0)$. In this section we will work with geodesic spherical coordinates $(r, \theta)$ with respect to the point $p$, see \cite[Sec.~III.1.]{Ch06}. 
Furthermore, we will assume that $(M,\mathbf{g})$ is spherically symmetric with respect to the point $p$. This means that the metric $\mathbf{g}$ can be written as \[\mathbf{g}=dr\otimes dr + f^2(r)\mathbf{g}_{\mathbb{S}^{n-1}}\] in geodesic spherical coordinates as long as $r_p<\rho$ for some fixed $\rho$. By \cite[Prop.~1.4.7]{Pe16} the function $f:[0,\rho)\to [0,\infty)$ satisfies 
\[f(0)=0,\quad f'(0)=1,\quad f''(0)=0.\]
 Many harmonic manifolds and surfaces of revolution are examples of spherically symmetric manifolds.
Writing the Laplacian in geodesic spherical coordinates gives
\[\Delta = \frac{\partial^2}{\partial r^2} + (n-1)\frac{f'(r)}{f(r)}\frac{\partial}{\partial r}+ \frac{1}{f^2(r)}\Delta_{\mathbb{S}^{n-1}_1}.\]
We refer the reader to \cite[Sec.~4.2.3]{Pe16} for more information about spherically symmetric manifolds.

Using separation of variables we will study families of Dirichlet eigenvalues on the ball $B_t(p)$. We will refer to the parametrized family $R^t(r)$ as the radial part of the solution and $\Theta(\theta)$ as the spherical part. We will show that there exists an $L^2(B_t(p))$-basis on the form $R^{t}(r)\Theta(\theta)$ consisting of Dirichlet Laplace eigenfunctions on the ball $B_{t}(p)$. The parametrized family $R^t$ will always be assumed to be continuous in $t$. In the case for the sphere this was shown in \cite[Chap.\ II.5]{Ch84} with a similar approach.
The radial part $R^{t}$ solves the equation
      \begin{equation}\label{eq:radial part_spherical}
        (f^{n-1}(r) (R^{t})'(r))'- f^{n-3}(r)m(m+n-2)R^{t}(r)=-\lambda(t) f^{n-1}(r)R^{t}(r),
      \end{equation}
      with $R^t(t)=0$ and where $R^t(0)$ is bounded.
      Introduce the norm 
      \[\|\phi\|^2_f=\int_0^{t}\phi(r)^2f^{n-1}(r)\,\mathrm{dr}.\]
      The operator 
      \[(L_m\phi)(r)=f^{1-n}(r)\left((f^{n-1}(r) \phi'(r))'- f^{n-3}(r)m(m+n-2)\phi(r)\right)\]
      is unbounded and symmetric on $L^2((0,t), f^{n-1}(r)\mathrm{dr})$ defined on the subspace 
      \[X_{t}= \left\{\phi \in H^2_{\mathrm{loc}}(0,t):\int_{0}^{t}f^{n-1}(r)\phi(r)^2\,\mathrm{dr}<\infty,\, \phi(t)=\phi(0)\phi'(0)=0\right\}.\]
      Recall that $H^2_{\mathrm{loc}}(0,t)$ is the space where the second weak derivatives are locally in $L^2(0,t)$.
      The eigenvalues of $L_m$ are simple, since if $u,v$ are two eigenfunctions with the same eigenvalue $\lambda$ we have that 
      \begin{align*}
          0&=f^{n-1}(r)(uL_mv-vL_mu)\\
          &=u(f^{n-1}(r)v')'-v(f^{n-1}(r)u')'\\
          &=(f^{n-1}(r)(uv'-vu'))'.
      \end{align*}
      This means that $f^{n-1}(r)(uv'-vu')$ is constant. Using that $u(t)=v(t)=0$ we get that the Wronskian $uv'-vu'$ is zero. Hence $u$ and $v$ are linearly dependent.
      
      We will let $\lambda_{m,l}(t)$ be the $l$'th eigenvalue of the problem \eqref{eq:radial part_spherical} with the corresponding eigenfunction $R^{t}_{m,l}$.
    \begin{proposition}
       Let $(M, \mathbf{g})$ be a spherically symmetric manifold with respect to the point $p$ and a constant $0<\rho$ and consider the ball $B_{t}(p) \subset M$ with $t<\rho$. Then there exists an $L^2$-basis on $B_{t}(p)$ which consists of Dirichlet Laplace eigenfunctions on the form $u^t(r,\theta)=R^{t}_{m,l}(r)\Theta_m(\theta),$ where 
       \[\Delta_{\mathbb{S}^{n-1}}\Theta_m+m(m+n-2)\Theta_m = 0,\]
       and $R_{m,l}^{t}$ is the eigenfunction corresponding to the $l$'th eigenvalue of \eqref{eq:radial part_spherical}.
    \end{proposition}
    \begin{proof}
 Let $u(r,\theta) = R(r) \Theta(\theta)$ be an eigenfunction written in geodesic spherical coordinates with eigenvalue $\lambda(t)$.
    Then we have
    \begin{align*}
        \Delta u(r,\theta) &= R''(r)\Theta (\theta)+ (n-1)\frac{f'(r)}{f(r)} R'(r)\Theta(\theta) +R(r) \frac{1}{f^2(r)}\Delta_{\mathbb{S}_1^{n-1}}\Theta(\theta)\\
        &= -\lambda(t) R(r)\Theta(\theta).
    \end{align*}
    The above expression simplifies to 
    \begin{equation}
    \label{eq:independent:sides}
        \frac{f^2(r)R''(r)+ (n-1)f'(r)f(r) R'(r)+\lambda(t) f^2(r)R(r)}{R(r)} = - \frac{\Delta_{\mathbb{S}_1^{n-1}}\Theta(\theta)}{\Theta(\theta)}.
    \end{equation}
    Since the left hand-side of \eqref{eq:independent:sides} is independent of $\theta$ we have that $\Theta$ is a spherical harmonic function. Using that the eigenvalues on the $(n-1)$-sphere have the form $m(m+n-2)$ we get that $R$ satisfies \eqref{eq:radial part_spherical}.
    Hence if $R(r)$ is a solution to \eqref{eq:radial_spherical} satisfying $R(t)=0$ and $R(0)$ is bounded, then $u$ is a solution to the Dirichlet Laplace eigenvalue problem. 

    The only thing left to show is that each eigenfunction can be written as a sum of eigenfunctions on the form $R(r)\Theta(\theta)$. Since the spherical harmonics are the eigenfunctions of $\mathbb{S}^{n-1}$, we know that the spherical harmonics form an orthonormal basis for $L^2(\mathbb{S}^{n-1})$. This means that an arbitrary eigenfunction $u(r,\theta)$ with eigenvalue $\lambda(t)$ can be written as \[u(r,\theta) = \sum_{i=1}^\infty a_i(r)\Theta_i(\theta),\]
    where $\Theta_i$ is the $i$'th spherical harmonic function with eigenvalue $m_i(m_i+n-2)$. Using the Laplacian written out in spherical coordinates gives
    \begin{align*}
        -m_i(m_i+n-2)a_i(r)&=\int_{\mathbb{S}^{n-1}_1}u(r,\theta)\Delta_{\mathbb{S}_1^{n-1}}\Theta_i(\theta)\dS\\
        &=\int_{\mathbb{S}^{n-1}_1}\Delta_{\mathbb{S}_1^{n-1}}u(r,\theta)\Theta_i(\theta)\dS\\
        &=-\lambda(t) f^2(r) a_i(r)-f^2(r)a''_i(r)-(n-1)f'(r)f(r)a_i'(r).
    \end{align*}
    In particular, the function $a_i(r)\Theta_i(\theta)$ is an eigenfunction for all $i$. By orthogonality of eigenfunctions we get that $u$ can be written on the form \[u(r,\theta)=\sum_{i=j}^{m}a_i(r)\Theta_i(\theta),\] where $a_i(r)\Theta_i(\theta)$ has eigenvalue $\lambda(t)$.

    \end{proof}
\section{Eigenvalues on Balls}\label{sec:eigenvalues_balls}

Let $(M,\mathbf{g})$ be a Riemannian manifold (not necessarily spherically symmetric) and consider the ball $B_t(p)\subset M$ for $p \in M$ and $t<\inj(p)$. The notation $\inj(p)$ denotes the injectivity radius at $p$. Again we will look at the Dirichlet problem 
\begin{equation}\label{D:hh}
   \begin{cases}
      \Delta u^t+\lambda\left( t \right)u^t =0&\text{ on }B_t(p)\\
      u^t=0&\text{ on }S_{t}^{n-1}(p)
   \end{cases}.
\end{equation}
We will assume that $\int_{B_t(p)}(u^t)^2\dvol = 1$, in which case 
\[\lambda(t)= \int_{B_t(p)}|\grad u^t|^2\dvol.\]
By using the Hadamard formula presented in \cite[Cor.~2.1]{EI07} one has, assuming that $\lambda(t)$ and $u^t$ are differentiable with respect to $t$, that
\begin{equation*}
   \lambda'\left( t \right)=-\int_{S_{t}^{n-1}(p)}\left|u^t_{n}\right|^{2}\dS.
\end{equation*}

For solutions to the Dirichlet problem we have the following result.
\begin{proposition}\label{lem:vari}
   Let $u^t$ be a parametrized family of solution to \eqref{D:hh} which is normalized in $L^2$. 
   Denote by $r$ the radial distance from the point $p$ and let $r_1<\inj(p)$. Assume that both $\lambda(t)$ and $u^t$ are differentiable in $t$ for $r_0\le t\le r_1$.
   Then
   \begin{multline*}
      r_1^2\lambda(r_1)=\lim_{t\to r_0}t^2\lambda(t)+\\
      \int_{r_0}^{r_1} t\int_{B_{t}(p)}(u^t)^{2}\frac{\Delta\left( r_p\Delta r_p \right)}{2} +2 \left|\grad_{S^{n-1}_{r_p}(p)} u^t\right|^{2}-2r_p\nabla^{2}r_p\left( \grad u^t,\grad u^t \right)\dvol \mathrm{dt}.
   \end{multline*}
\end{proposition}
The above proposition was proved for the first eigenvalue in \cite[Lem.~3.1]{BF17}. The proof was based on variational methods. We will give another proof for general eigenvalues.
We first develop the following lemma of independent interest.
\begin{lemma}\label{lem:1_D}
    Let $u$ be a solution to $\Delta u + \lambda u = 0$ on the ball $B_t(p)$ where $t\le \inj(p)$.
    Let $\varphi:(0, \infty) \to (0, \infty)$ be such that $\frac{\grad r_p}{\varphi\left( r_p\left( x \right) \right)}$ is a smooth vector field. Then 
    \begin{align*}
        &\int_{S_{t}^{n-1}(p)}|\grad_{S_{t}^{n-1}(p)} u|^2-u_{n}^{2}\dS\\
        &=\varphi\left( t \right) \int_{B_{t}(p)}|\grad u|^{2}\left( \frac{\varphi\left( r_p\left( x \right) \right)\Delta r_p-\varphi'\left( r_p\left( x \right) \right)}{\varphi^{2}\left( r_p\left( x \right) \right)} \right)\dvol\\
        &\quad-2\varphi\left( t \right) \int_{B_{t}(p)}\frac{\left(\nabla^{2}r_p\left( \grad u,\grad u \right) -\lambda u_{r}u\right)\varphi\left( r_p\left( x \right) \right)-\varphi'\left( r_p\left( x \right) \right)u_r^{2}}{\varphi^{2}\left( r_p\left( x \right) \right)}\dvol.
    \end{align*}
\end{lemma}
\begin{proof}
   We will use the notation $X=\frac{\grad r_p}{\phi(r_p(x))}$ and \[V=2X(u)\grad u-|\grad u|^2X.\]
   Taking the divergence of $V$, one obtains
   \begin{align*}
      \dv(V)&=2\langle\grad X(u),\grad u\rangle-2\lambda uX(u)-\langle\grad |\grad u|^2,X\rangle\\ 
      &\quad-|\grad u|^2\dv(X)\\
      &=2\langle \nabla_{\grad u} X,\grad u\rangle + 2\langle X,\nabla_{\grad u}\grad u\rangle-2\langle\nabla_X\grad u,\grad u\rangle \\
      &\quad -2\lambda uX(u)-|\grad u|^2\dv(X)\\
      &=2\langle \nabla_{\grad u} X,\grad u\rangle+\nabla^2u(\grad u,X) -2\nabla^2u(X,\grad u)-2\lambda uX(u)\\
      &\quad-|\grad u|^2\dv(X)\\
      &=2\langle \nabla_{\grad u} X,\grad u\rangle-2\lambda uX(u)-|\grad u|^2\dv(X).
   \end{align*}
   Expanding the term 
   \[\left\langle \nabla_{\grad u}\frac{\grad r_p}{\phi(r_p(x))},\grad u\right\rangle=\frac{\phi(r_p(x))\nabla^2r_p(\grad u,\grad u)-\phi'(r_p(x))u_r^2}{\phi(r_p(x))^2}\]
   we get
   \begin{align*}
      \dv(V)&=\frac{2(\phi(r_p(x))\nabla^2r_p(\grad u,\grad u)-\phi'(r_p(x))u_r^2)-2\lambda\phi(r_p(x))uu_r}{\phi(r_p(x))^2}\\ 
      &\quad+\frac{|\grad u|^2(\phi(r_p(x))\Delta r_p-\phi'(r_p(x)))}{\phi(r_p(x))^2}.
   \end{align*}
   Using the divergence theorem with $\dv(V)$ together with  \[\langle V, \grad r_p\rangle=\frac{|\grad_{S_t^{n-1}(p)}u|^2-(u_r)^2}{\phi(t)},\]
   gives the result.
\end{proof}
\begin{proof}[Proof of Prop.~\ref{lem:vari}]
Lemma \ref{lem:1_D} with $\varphi(r_p(x)) =\frac{1}{ r_p(x)}$ implies that
\begin{multline*}
   t\lambda'\left( t \right)=\int_{B_{t}(p)}\Big(\left|\grad u^t\right|^{2}\left( r_p\Delta r_p +1 \right)\\
   -2\left( r_p\nabla^{2}r_p\left( \grad u^t,\grad u^t \right)+(u^t)^{2}_{n}-\lambda\left( t \right)r_p(u^t)_{n}u^t \right)\Big)\dvol.
\end{multline*}
Using the equation
\begin{equation*}
   \int_{B_{t}(p)}(u^t)_{n}u^tr_p\dvol=-\frac{1}{2}\int_{B_{t}(p)}(u^t)^{2}\dvol-\frac{1}{2}\int_{B_{t}(p)}r_p\Delta r_p (u^t)^{2}\dvol
\end{equation*}
gives
\begin{align}\label{eq:lambda prime}
   t\lambda'\left( t \right)&=\int_{B_{t}(p)}\left|\grad u^t\right|^{2}\left( r_p\Delta r_p +1 \right)-2\left( r_p\nabla^{2}r_p\left( \grad u^t,\grad u^t \right)+(u^t)^{2}_{n} \right)\dvol\nonumber\\
   &\quad-\lambda\left( t \right)-\lambda\left( t \right)\int_{B_{t}(p)}r_p\Delta r_p (u^t)^{2}\dvol.
\end{align}
The equation \eqref{eq:lambda prime} together with $(t^2\lambda(t))'=t(2\lambda(t)+t\lambda'(t))$ implies that 
\begin{align*}
   \left(t^2\lambda\left( t \right)\right)'&=t\int_{B_{t}(p)}\left|\grad u^t\right|^{2}\left( r_p\Delta r_p +2 \right)-2r_p\nabla^{2}r_p\left( \grad u^t,\grad u^t \right)\dvol\\ &\quad-2t \int_{B_t(p)}(u^t)^{2}_{n} \dvol-t\lambda\left( t \right)\int_{B_{t}(p)}r_p\Delta r_p (u^t)^{2}\dvol.
\end{align*}
Simplifying the expression further gives us that
\begin{align*}
   (t^2\lambda\left( t \right))'
   &=t\int_{B_{t}(p)}\left(\frac{1}{2}\Delta ((u^t)^2)+\lambda(t)(u^t)^2\right)r_p\Delta r_p\dvol \\
   &\quad+2t\int_{B_t(p)}\left( \left|\grad_{S^{n-1}_r(p)} u^t\right|^{2}-r_p\nabla^{2}r_p\left( \grad u^t ,\grad u^t \right)\right)\dvol\\
   &\quad-t\lambda\left( t \right)\int_{B_{t}(p)}r_p\Delta r_p (u^t)^{2}\dvol\\
   &=-t\lambda\left( t \right)\int_{B_{t}(p)}\frac{1}{2}\left\langle\grad (u^t)^{2},\grad\left( r_p\Delta r_p \right)\right\rangle\dvol\\
   &\quad +2t\int_{B_t(p)}\left( \left|\grad_{S^{n-1}_r(p)} u^t\right|^{2}-r_p\nabla^{2}r_p\left( \grad u^t,\grad u^t \right)\right)\dvol\\
   &=t\int_{B_{t}(p)}(u^t)^{2}\frac{\Delta\left( r_p\Delta r_p \right)}{2} +2 \left|\grad_{S^{n-1}_r(p)} u^t\right|^{2}\dvol\\
      &\quad-2t\int_{B_{t}(p)}r_p\nabla^{2}r_p\left( \grad u^t,\grad u^t \right)\dvol.
\end{align*}
Finally, integrating the identity above gives the result.
\end{proof}
\begin{remark}
Let $(M,\mathbf{g})$ be an analytic Riemannian manifold and let $\phi_s:B_t(p)\to B_{s+t}(p)$ be the flow of the vector field $\grad r_p$ which is analytic for $r+t< \inj(p)$. Then by \cite[Lem.~3.1]{EI07} we can find a differentiable family $\lambda(t)$ and $u^t$ for $t\in(r-\epsilon, r+\epsilon)$.
\end{remark}
\section{Spherical Symmetric Manifolds}\label{sec:Spherical Symmetric Manifolds}

In this section we will assume that the $n$-dimensional Riemannian manifold $(M,\mathbf{g})$ is spherically symmetric with respect to the point $p \in M$ on the ball $B_\rho(p)$. We will use the notation introduced in Section \ref{sec:spherically_symmetric}.
Before stating the main result, we will show that the following Hadamard formula holds:
\begin{theorem}\label{lem:der}
   Let $u^t(r,\theta)=R_{m,l}^t(r)\Theta(\theta)$ be an $L^2$-normalized solution to \eqref{D:hh} on the ball $B_t(p)$ where \[\Delta_{\mathbb{S}_1^{n-1}} \Theta=-m(m+n-2)\Theta,\] 
   and $R_{m,l}^t$ satisfies \eqref{eq:radial part_spherical} with $R_{m,l}^t(t)=0$ and $R_{m,l}^t(0)$ being bounded. 
   Then we have that 
   \[\frac{d\lambda_{m,l}(t)}{dt}=-f(t)^{n-1}(R_{m,l}^t(t)')^2, \qquad \text{a.e.\ in }t\in(0,\rho).\]
\end{theorem}
To show Theorem \ref{lem:der} we need the following result:
\begin{theorem}[{\cite[Thm.~2.1~p.~23]{We87}}]\label{thm:second_order}
Let $A_{\_}:(-c,c)\times(a,b)\to \mathrm{Mat}(\mathbb{R}^n)$ and $y_{\_}:(-c,c)\to \mathbb{R}^n$ be continuous functions. Fix a point $x_0\in (a,b)$.
Then the problem 
\[\begin{cases}
\frac{dU_h(r)}{dr}=A_h(r)U_h(r)&\text{ for } r\in (a,b),\\
U_h(x_0)=y_h,
\end{cases}\]
has a unique solution $U_h(x)$ that is continuous in $(-c,c)\times(a,b)$.
\end{theorem}
\begin{proof}
   For completeness, we repeat and  slightly modify the proof from \cite[Thm.~2.1~p.~23]{We87}.
   
   Fix $h$ and let $[x_0-\eta,x_0+\eta]$ be such that 
   \[\int_{x_0-\eta}^{x_0+\eta}|A_h|\,\mathrm{dt}\le q<1.\] Define the operator
   $B_h:C([x_0-\eta,x_0+\eta])\to C([x_0-\eta,x_0+\eta])$
   by 
   \[B_hu(x)=y_h+\int_{x_0}^xA_h(t)u(t)\,\mathrm{dt},\]
   where $C([x_0-\eta,x_0+\eta])$ is the space of continuous functions with the supremum norm $\|\cdot\|_{\infty}$.
   Notice that $B_h$ is a contraction since
   \[\|B_hu-B_hv\|_{\infty}\le q\|u-v\|_{\infty} .\]
   By the Banach fixed point theorem there exists a unique fixed point, which we will denote by $U_h$, giving the existence and uniqueness at the interval $[x_0-\eta, x_0+\eta]$. Choosing a compact set $K$ inside $(a,b)$ and a finite covering $\{(x_i-\eta,x_i+\eta)\}_{i\in \mathcal{J}}$ of $K$ where 
   \[\int_{x_i-\eta}^{x_i+\eta}|A_h|\,\mathrm{dt}\le q<1,\]
   gives the existence and uniqueness of the solution on $K$. A compact sweeping of the interval $(a,b)$ gives the uniqueness and existence result on $(a,b)$.
   
   We are only left to show that $U_h$ is continuous. Fix a value $h'\in (-c, c)$ and let $I= [x_0-\eta,x_0+\eta] \subset  (a,b)$ be such that 
   \[\int_{I}|A_{h'}(t)|\, \mathrm{dt}\le \frac{1}{2}.\] 
   Then for $y\in I$ we obtain
   \begin{align*}
      \sup_{x\in I}|U_{h'}(y)-U_h(x)| &=\sup_{x\in I}|B_{h'}U_{h'}(y)-B_hU_h(x)|\\
      &\le|y_{h'}-y_h| +|x-y|\sup_{x\in I}|A_{h'}(x)|\sup_{x\in I}|U_{h'}(x)|\\
      &\quad +\sup_{x\in I}|U_{h'}(y)-U_h(x)|\int_{I}|A_{h}(t)|\, \mathrm{dt}\\
      &\quad +\sup_{x\in I}|U_{h'}(y)-U_{h'}(x)|\int_{I}|A_{h}(t)|\, \mathrm{dt}\\
      &\quad+ \sup_{x\in I}|U_{h'}(x)|\int_{I}|A_h(t)-A_{h'}(t)|\, \mathrm{dt}.
   \end{align*}
  One can find a box $[h'-\xi, h'+\xi]\times [y-\delta, y+\delta]$ such that for all $x\in [y-\delta, y+\delta]$ and all $h\in [h'-\xi, h'+\xi]$ we have that
   \[\sup_{x\in I}|U_{h'}(x)|\sup_{h\in [h'-\xi, h'+\xi]}\int_{I}|A_h(t)-A_{h'}(t)|\, \mathrm{dt}\le \frac{\epsilon}{8},\]
   \[\sup_{x\in I}|U_{h'}(y)-U_{h'}(x)|\int_{I}|A_{h}(t)|\, \mathrm{dt}\le \frac{\epsilon}{8},\]
  \[|x-y|\sup_{x\in I}|A_{h'}(x)|\sup_{x\in I}|U_{h'}(x)|<\frac{\epsilon}{8},\] 
   and 
   \[\sup_{h\in [h'-\xi, h'+\xi]}|y_h-y_{h'}|\le \frac{\epsilon}{8},\]
   showing continuity at $(h',y)$.
   Hence $U_{h}(x)$ is continuous.
\end{proof}
\begin{proof}[Proof of Thm.~\ref{lem:der}]
The outline of the proof is inspired by \cite[Thm.~3.1]{KZ96}, where the author shows several similar results for Sturm-Liouville equations.

Recall that $R_{m,l}^t$ satisfies the equation 
\begin{equation*}
\begin{cases}
   (f^{n-1}(r)(R_{m,l}^t(r))')'+f^{n-1}(r)\left(\lambda_{m,l}(t)-\frac{m(m+n-2)}{f^{2}(r)}\right)R_{m,l}^t(r)=0,\\
   R_{m,l}^t(0)(R_{m,l}^t)'(0)=R_{m,l}^t(t)=0.
\end{cases}
\end{equation*}
The solution is assumed to satisfy 
   \[\int_0^tf^{n-1}(r)(R_{m,l}^t(r))^2\,\mathrm{dr}=1.\]
  We will start by showing that the family  $R_{m,l}^t$ is uniformly continuous on compact subsets of $(0,t]$ . 
   The differential equation 
   \[   (f^{n-1}(r)(R_{m,l}^t(r))')'+f^{n-1}(r)\left(\lambda_{m,l}(t)-\frac{m(m+n-2)}{f^{2}(r)}\right)R_{m,l}^t(r)=0\]
   has two linearly independent solutions on the interval $(0,t]$, one of which is bounded at $0$ and the other is unbounded at $0$. By existence and uniqueness of second order ordinary differential equations, see e.g.\ \cite[Thm.~A~p.~488]{Si17}, the two solutions can not have common zeroes. Additionally, we will use that the eigenvalues are continuous in the variable $t$, see e.g.~\cite[Thm.~2.10]{Ur17}.  
   
  Let $0\le h<\epsilon_1$, where $t+\epsilon_1<\rho$, and $\rho$ is the largest radius where the manifold is radially symmetric. Define the matrix 
  \[A_h(r)=\frac{t+h}{t}\begin{pmatrix}
      0 & f^{1-n}(\frac{t+h}{t}r)\\
      f^{n-1}(\frac{t+h}{t}r)\left(\frac{m(m+n-2)}{f^{2}(\frac{t+h}{t}r)}-\lambda_{m,l}(t+h)\right)& 0
  \end{pmatrix}.
\]
Then we can consider the problem
  \begin{equation}\label{eq:onedim}
  \begin{cases}
  U_h'(r) =A_h(r)U_h(r)& \text{on } (0,t\rho/(t+\epsilon_1)),\\
  U_h(t)=\begin{pmatrix}
      0\\
      \sqrt{\frac{t}{t+h}}
   \end{pmatrix}.
  \end{cases}
\end{equation}
By Theorem \ref{thm:second_order} this problem has a unique continuous solution coinciding with 
   \[U_h(r)=\frac{1}{f^{n-1}(t+h)(R_{m,l}^{t+h})'(t+h)}\begin{pmatrix}
      \sqrt{\frac{t}{t+h}}R_{m,l}^{t+h}(\frac{t+h}{t}r)\\
      \sqrt{\frac{t}{t+h}}f^{n-1}(\frac{t+h}{t}r)(R_{m,l}^{t+h})'(\frac{t+h}{t}r)
   \end{pmatrix},\]
   on the interval $(0,t]$.
   This means that the components of $U_h$ given by
   \[U_{h,1}(r)=\frac{\sqrt{\frac{t}{t+h}}R_{m,l}^{t+h}(\frac{t+h}{t}r)}{f^{n-1}(t+h)(R_{m,l}^{t+h})'(t+h)}\]
   and
   \[ U_{h,2}(r)=\frac{ \sqrt{\frac{t}{t+h}}f^{n-1}(\frac{t+h}{t}r)(R_{m,l}^{t+h})'(\frac{t+h}{t}r)}{f^{n-1}(t+h)(R_{m,l}^{t+h})'(t+h)}\]
   are continuous.
   
  We will now turn to finding the derivative of $\lambda_{m,l}(t)$. Denote by \[N(t,h) = \frac{t/(t+h)}{f^{n-1}(t+h)(R_{m,l}^{t+h})'(t+h)f^{n-1}(t)(R_{m,l}^{t})'(t)}. \]
  Using integration by parts shows that 
  \begin{align*}
  U_{h,1}\left(\frac{t^2}{t+h}\right)U_{0,2}(t)&=N(t,h)\int_0^t(R_{m,l}^{t+h}(r)f^{n-1}(r)(R_{m,l}^{t})'(r))'\,\mathrm{dr} \\\nonumber 
   &=N(t,h)\int_0^tR_{m,l}^{t+h}(r)(f^{n-1}(r)(R_{m,l}^{t})'(r))'\,\mathrm{dr}\\\nonumber
   &\quad-N(t,h)\int_0^t((R_{m,l}^{t+h})'(r)f^{n-1}(r))'R_{m,l}^{t}(r)\,\mathrm{dr}.\\\nonumber
   \end{align*}
   Using the differential equation for $R^{t+h}_{m,l}$ and $R^{t}_{m,l}$ gives
   \begin{multline}
   \label{eq:rew1}
   U_{h,1}\left(\frac{t^2}{t+h}\right)U_{0,2}(t)\\=(\lambda_{m,l}(t+h)-\lambda_{m,l}(t))N(t,h)\int_0^tf^{n-1}(r)R_{m,l}^{t+h}(r)R_{m,l}^{t}(r)\,\mathrm{dr}.
   \end{multline}
  
  Let us show that 
  \begin{equation}\label{eq:limit one}
  \lim_{h\to 0} N(t,h)\int_0^tf^{n-1}(r)R_{m,l}^{t+h}(r)R_{m,l}^{t}(r)\,\mathrm{dr}=N(t,0).
  \end{equation}
  By the continuity of $U_h(x)$ we immediately have for $0<\epsilon<t$ the limit
  \begin{align*}
  \lim_{h\to 0}N(t,h)\int_\epsilon^tf^{n-1}(r)R_{m,l}^{t+h}(r)R_{m,l}^{t}(r)\,\mathrm{dr}&=N(t,0)\int_\epsilon^{t} f^{n-1}(r)R_{m,l}^{t}(r)^2\,\mathrm{dr}.
  \end{align*}
  For the limit
  \[\lim_{h\to 0} N(t,h)\int_0^\epsilon f^{n-1}(r)R_{m,l}^{t+h}(r)R_{m,l}^{t}(r)\,\mathrm{dr}\]
  we will show that the integrand $f^{n-1}(r)R_{m,l}^{t+h}(r)R_{m,l}^{t}(r)$ is uniformly bounded. In the case that $m>1$ we have that $R_{m,l}^{t+h}(r)$, $f^{n-1}(r)$, and $R_{m,l}^{t}(r)$ are increasing functions for small $r$. To be more precise, fix $r_0$ such that for all $r<r_0$ and all $h\in (-c,c)$ we have that $f'(r)>0$ and 
  \[\min_{h\in (-c,c)}\lambda_{m,l}(t+h)-\frac{m(m+n-2)}{f^2(r)}<0.\]
  Then 
  \[f^{n-1}(r)R_{m,l}^t(r)R_{m,l}^{t+h}(r)\le f^{n-1}(r_0)\max_{h\in (-c,c)}(R_{m,l}^{t+h}(r_0))^2. \]
  To get this result, notice that \eqref{eq:radial part_spherical} implies that $(R^{t+h}_{m,l})'(r)$ can not be zero as long as  
  \[\lambda(t+h)-\frac{m(m+n-2)}{f^2(r)}<0,\]
  by the second derivative test.
  Choosing $\epsilon<r_0$ we get by the dominated convergence theorem that 
    \begin{align*}
  \lim_{h\to 0}N(t,h)\int_0^\epsilon f^{n-1}(r)R_{m,l}^{t+h}(r)R_{m,l}^{t}(r)\,\mathrm{dr}&=N(t,0)\int_0^\epsilon f^{n-1}(r)R_{m,l}^{t}(r)^2\,\mathrm{dr}
  .
  \end{align*}

  In the case that $m=0$, we fix $r_0$ such that for all $r<r_0$ we have that $R^{t}_{0,l}(r)>0$ and $f'(r)>0$. Then 
  \[f^{n-1}(r)R^{t+h}_{0,l}(r)R^t_{0,l}(r)\le  R^{t+h}_{0,l}(r)f^{n-1}(r)\max_{x\in [0,r_0]}R^{t}_{0,l}(x).\]
  Using integration by parts and the Cauchy-Schwarz inequality we obtain
  \begin{align*}
  f^{n-1}(r)R^{t+h}_{0,l}(r)&=f^{n-1}(r)R^{t+h}_{0,l}(r)-f^{n-1}(t+h)R^{t+h}_{0,l}(t+h)\\
  &=-\int_{r}^{t+h}(n-1)f'(s)f^{n-2}(s)R^{t+h}_{0,l}(s)\,\mathrm{ds}-\int_{r}^{t+h}f^{n-1}(s)R^{t+h}_{0,l}(s)'\,\mathrm{ds}\\
  &\le (n-1)\sqrt{\int_{r}^{t+h}f'(s)^2f^{n-3}(s)\,\mathrm{ds}\int_{0}^{t+h}R^{t+h}_{0,l}(s)^2f^{n-1}(s)\,\mathrm{ds}}\\
  \quad &+\sqrt{\int_{0}^{t+h}f^{n-1}(s)(R^{t+h}_{0,l}(s)')^2\,\mathrm{ds}\int_{r}^{t+h}f^{n-1}(s)\,\mathrm{ds}}\\
  &\le (n-1)\sqrt{\int_{r}^{t+h}f'(s)^2f^{n-3}(s)\,\mathrm{ds}}+\sqrt{\lambda_{0,l}(t+h)\int_{r}^{t+h}f^{n-1}(s)\,\mathrm{ds}}.
  \end{align*}
  Hence we can use the dominated convergence theorem once more and obtain \eqref{eq:limit one}.
  
  Set \[v(t,h)=\frac{\sqrt{t/(t+h)}}{f^{n-1}(t+h)(R_{m,l}^{t+h})'(t+h)}.\]
  Rewriting \[U_{h,1}(t^2/(t+h))=v(t,h)R_{m,l}^{t+h}(t)\] gives
  \begin{align*}
  v(t,h)R_{m,l}^{t+h}(t)&=v(t,h)R_{m,l}^{t+h}(t)-v(t,h)R_{m,l}^{t+h}(t+h)\\\nonumber
  &=-\int_t^{t+h}\frac{1}{f^{n-1}(r)}(v(t,h)f^{n-1}(r)(R_{m,l}^{t+h})'(r))\,\mathrm{dr}\\\nonumber
  &=\int_t^{t+h}\frac{v(t,h)f^{n-1}(r)(R_{m,l}^{t})'(r)-v(t,0)f^{n-1}(r)(R_{m,l}^{t+h})'(r)}{f^{n-1}(r)}\,\mathrm{dr}\\\nonumber
  &\qquad -\int_t^{t+h}\frac{1}{f^{n-1}(r)}(v(t,0)f^{n-1}(r)(R_{m,l}^{t})'(r))\,\mathrm{dr}.\nonumber
  \end{align*}
If we insert the above equation into \eqref{eq:rew1} and divide by $h$ we get
\begin{align*}
   &\frac{\lambda_{m,l}(t+h)-\lambda_{m,l}(t)}{h}N(t,h)\int_0^tf^{n-1}(r)R_{m,l}^{t+h}(r)R_{m,l}^{t}(r)\,\mathrm{dr}\\
  &=-v(t,0)f^{n-1}(t)(R_{m,l}^{t})'(t)\Big(\frac{1}{h}\int_t^{t+h}\frac{1}{f^{n-1}(r)}(v(t,0)f^{n-1}(r)(R_{m,l}^{t})'(r))\,\mathrm{dr} \\
  &\quad-\frac{1}{h}\int_t^{t+h}\frac{v(t,h)f^{n-1}(r)(R_{m,l}^{t})'(r)-v(t,0)f^{n-1}(r)(R_{m,l}^{t+h})'(r)}{f^{n-1}(r)}\,\mathrm{dr}\Big).
\end{align*}
Taking the limit as $h\to 0$ and using the continuity of $U_{h,2}$ gives the result.
\end{proof}
Using Proposition \ref{lem:vari} we get the following corollary.  
\begin{corollary}\label{D:cor:1}
   Let $u^t(r,\theta)=R_{m,l}^t(r)\Theta(\theta)$ be a solution to \eqref{D:hh} on the ball $B_t(p)$ where \[\Delta_{\mathbb{S}_1^{n-1}} \Theta=-m(m+n-2)\Theta,\] 
   and $R_{m,l}^t$ satisfies \eqref{eq:radial part_spherical} with $R_{m,l}^t(t)=0$ and $R_{m,l}(0)$ being bounded. 
   Denote by 
   \begin{align*}
      F(r)&=\frac{n-1}{f^3(r)}\Big((3-n)rf'(r)^3+(rf'''(r)+2f''(r))f^2(r)\\
      &\qquad+((n-4)rf''(r)+(n-3)f'(r))f'(r)f(r)\Big).
   \end{align*}
    Then for $0 < r_{0} < r_1 <\rho$ we have
   \begin{align*}
   \lambda_{m,l}(r_1)&=\frac{1}{2r_1^2}\int_{r_0}^{r_1}t\int_{0}^tR_{m,l}^t(r)^{2}\left(F(r)+\frac{4m(m-2+n)}{f^2(r)}\left(1 -r\frac{f'(r)}{f(r)}\right)\right)f^{n-1}(r)\,\mathrm{dr}\, \mathrm{dt}\\
   &\quad +\frac{1}{r_1^2}\lim_{s\to r_0}s^2\lambda_{m,l}(s).
   \end{align*}
   In particular taking the limit of $r_0\to 0$ one has that 
   \begin{align*}
   \lambda_{m,l}(r_1)&=\frac{1}{2r_1^2}\int_{0}^{r_1}t\int_{0}^tR_{m,l}^t(r)^{2}\left(F(r)+\frac{4m(m-2+n)}{f^2(r)}\left(1 -r\frac{f'(r)}{f(r)}\right)\right)f^{n-1}(r)\,\mathrm{dr}\,\mathrm{dt}\\
   &\quad +\frac{(j_{m+n/2-1}^l)^2}{r_1^2}.
   \end{align*}
\end{corollary}
\begin{remark}
   For the first eigenvalue, i.e.~$m=0$ and $l=0$, the corollary was shown in \cite[Lem.~3.1]{BF17}.
\end{remark}
\begin{proof}
Using Proposition \ref{lem:vari} we need to compute $\Delta(r_p\Delta r_p)$. Notice first that \[\Delta r_p = (n-1)\frac{f'(r)}{f(r)}.\] Moreover, the radial part of the Laplacian is $\partial_r^2+(n-1)\frac{f'(r)}{f(r)}\partial_r$.
Hence we get that 
\[\Delta(r_p\Delta r_p)=\left(\partial_r^2+(n-1)\frac{f'(r)}{f(r)}\partial_r\right)\frac{(n-1)rf'(r)}{f(r)}=F(r).\]
Using the divergence theorem with $s<t$ we get that 
\[\int_{S^{n-1}_s}|\grad_{S^{n-1}_s} u^t|^2\dS=\frac{m(m+n-2)}{f^2(s)}\int_{S^{n-1}_s}(u^t)^2\dS,\]
where we have used that $\Delta_{S_s^{n-1}}=\frac{1}{f^2(s)}\Delta_{\mathbb{S}^{n-1}_1}$. Thus the result follows.
\end{proof}

\section{Second Proof of Theorem \ref{thm:radial:model:space}}\label{ex:model_eigenvalues}
As stated in the introduction of the article, we can use Corollary \ref{D:cor:1} to give a new proof of Theorem \ref{thm:radial:model:space}. 
\begin{proof}[Proof of Thm.~\ref{thm:radial:model:space}]
   In the case that $(M_K,\mathbf{g}_K)$ is a model space we have that $f(r)=\sin_K(r)$. Hence the function $F$ in Corollary \ref{D:cor:1} simplifies to 
   \[F(r)=(-K)(n-1)^2+(n-1)(n-3)\frac{\sin_K(r)-r\cos_K(r)}{\sin_{K}^3(r)}.\]
   When $K>0$ we will assume that $r_1<\frac{\pi}{\sqrt{K}}$.
   The function
   \[G(r) = \frac{\sin_K(r)-r\cos_K(r)}{\sin_{K}^3(r)}\]
   is increasing and satisfies $\lim_{r\to 0}G(r) = \frac{-K}{3}$.
   Thus for all $r\in (0,t]$ we have
   \[\frac{-K}{3}\le G(r)\le \frac{\sin_K(t)-t\cos_K(t)}{\sin_{K}^3(t)}.\]

   We will assume that $\lambda_{m,l}(t)$ is the $l$'th eigenvalue of $R_{m,l}^t$ and that $u^t(r,\theta)=R_{m,l}^t(r)\Theta_m(\theta)$ where $u^t$ is normalized in $L^2(B_t(p))$ and $R_{m,l}^t$ satisfies \eqref{eq:radial part_spherical} and $\Theta_m$ is a spherical harmonic function with eigenvalue $m(m+n-2)$.
   When $n\ge  3$ or $m\ge 1$ we get that  
   \begin{align*}
   \lambda_{m,l}(r_1)&\le\frac{(j^l_{m+n/2-1})^2}{r_1^2}-K\frac{(n-1)^2}{2r_1^2}\int_0^{r_1}t \, \mathrm{dt}\\
   &\qquad+\frac{(n-1)(n-3)+4m(m+n-2)}{2r_1^2}\int_{0}^{r_1}\left[\frac{t^2}{\sin_K^2(t)}\right]'\mathrm{dt}\\
   &=\frac{(j^l_{m+n/2-1})^2}{r_1^2}-K\frac{(n-1)^2}{4}+\frac{4(m + (n-2)/2)^2-1}{4}\left( \frac{1}{\sin_{K}^2(r_1)}-\frac{1}{r_1^2} \right).
   \end{align*}
    This is the same inequality as in Theorem \ref{thm:radial:model:space}.
   Using the lower bound of $G(t)$ we obtain
   \begin{equation*}
   \lambda_{m,l}(r_1)\ge\frac{(j^l_{m+n/2-1})^2}{r_1^2}+\frac{2m^2+2m(n-2)-n(n-1)}{6}K.
   \end{equation*}
   
   For $n=2$ and $m=0$ we have that
   \[-K-\frac{\sin_K(t)-t\cos_K(t)}{\sin_{K}^3(t)}\le F(r)=-K-\frac{\sin_K(r)-r\cos_K(r)}{\sin_{K}^3(r)}\le -\frac{2K}{3}.\] In this case, we have that the eigenvalue $\lambda_{0,l}(r_{1})$ satisfies
   \begin{align*}
   \lambda_{0,l}(r_1)&=\frac{1}{2r_1^2}\int_{0}^{r_1}t\int_{0}^tR_{0,l}^t(r)^{2}F(r)\sin_K(r)\,\mathrm{dr}\,\mathrm{dt}+\frac{(j_{0}^l)^2}{r_1^2}.
   \end{align*}
   Using the upper bound on $F$ we get
   \[\lambda_{0,l}(r_1)\le \frac{-K}{6}+\frac{(j_0^l)^2}{r_1^2}.\]
   For the lower bound on $F$ we have
   \[\frac{-K}{4}+\frac{1}{4}\left(\frac{1}{r_1^2}-\frac{1}{\sin_K^2(r_1)}\right)\le \lambda_{0,l}(r_1).\qedhere\]
\end{proof}

\bibliographystyle{abbrv}
\bibliography{references}
\end{document}